\documentclass{amsart}
\usepackage{amssymb}
\usepackage{amsfonts}
\usepackage{amsmath}
\usepackage{graphicx}
\usepackage{shadow}
\usepackage{color}
\usepackage[all]{xy}

\usepackage[pagebackref]{hyperref}

\newtheorem{thm}{Theorem}[section]

\newtheorem{corollary}[thm]{Corollary}
\newtheorem{lemma}[thm]{Lemma}
\newtheorem{proposition}[thm]{Proposition}

\theoremstyle{definition}
\newtheorem{definition}[thm]{Definition}

\theoremstyle{remark}
\newtheorem{remark}[thm]{Remark}

\DeclareMathOperator{\id}{id}

\DeclareMathOperator{\Ker}{Ker}

\DeclareMathOperator{\m}{\frak{m}}
\DeclareMathOperator{\p}{\frak{p}}

\DeclareMathOperator{\Hom}{Hom}
\DeclareMathOperator{\Ext}{Ext}

\DeclareMathOperator{\Tor}{Tor}

\def\pd{{\rm pd}}
\def\fd{{\rm fd}}

\def\id{{\rm id}}

\def\w-wgldim{{\rm w-w.gl.dim}}

\def\sup{{\rm sup}}

\def\FP{{\rm FP}}
\def\fpd{{\rm fpd}}
\def\fpD{{\rm fpD}}

\everymath{\displaystyle\everymath{}}
\begin{document}
	
	\title[$w$-FP-PROJECTIVE MODULES AND DIMENSION]{$w$-FP-PROJECTIVE MODULES AND DIMENSION}

	\author[R.A.K. Assaad]{Refat Abdelmawla Khaled Assaad}
	\address{Department of Mathematics, Faculty of Science, University Moulay Ismail Meknes, Box 11201, Zitoune, Morocco}
	\email{refat90@hotmail.com}
	\author[E. M. Bouba]{El Mehdi Bouba}
	\address{Department of Mathematics and computer, Faculty Multidisciplinary of Nador,University Mohammed first, Morocco}
	\email{mehdi8bouba@hotmail.fr}
	\author[M. Tamekkante]{Mohammed Tamekkante}
	\address{Department of Mathematics, Faculty of Science, University Moulay Ismail Meknes, Box 11201, Zitoune, Morocco}
	\email{tamekkante@yahoo.fr}
	
	\begin{abstract}
		Let $R$ be a ring. An $R$-module $M$ is said to be an absolutely $w$-pure module if and only if  $\Ext^1_R(F,M)$ is a GV-torsion module for any finitely presented module $F$. In this paper, we introduce and study the concept of $w$-FP-projective module which is in some way a generalization of the notion of $\FP$-projective module. An $R$-module $M$ is said to be $w$-$\FP$-projective if $\Ext^1_R(M,N)=0$ for any absolutely $w$-pure module $N$. This new class of modules will be used to characterize (Noetherian)
		$DW$ rings. Hence, we introduce the $w$-$\FP$-projective dimensions of modules and rings. The relations between the introduced dimensions and other (classical) homological dimensions are discussed. Illustrative examples  are given.
	\end{abstract}
	\maketitle
	
	\section{\bf Introduction}
	Throughout, all rings considered are commutative with unity and all modules are unital. Let $R$ be a ring and $M$ be an $R$-module. As usual, we use ${\rm pd}_R(M)$, ${\rm id}_R(M)$, and ${\rm fd}_R(M)$ to denote, respectively, the classical projective dimension, injective dimension, and flat dimension of $M$, and ${\rm wdim}(R)$ and ${\rm gldim}(R)$ to denote, respectively, the weak and global homological dimensions of $R$.\\
	Now, we review some definitions and notation. Let $J$ be an ideal of $R$. Following \cite{HFX}, $J$ is called a \emph{Glaz-Vasconcelos ideal} (a $GV$-ideal for short) if $J$ is finitely generated and the natural homomorphism $\varphi : R \rightarrow J^{\ast} = {\rm Hom}_R(J,R)$ is
	an isomorphism.  Let $M$ be an $R$-module and define
	$${\rm tor}_{GV}(M) = \{x\in M \mid Jx = 0\text{  for some } J\in  GV(R)\},$$
	where  $GV(R)$ is the set of $GV$-ideals of $R$. It is clear that ${\rm tor}_{GV}(M)$ is a submodule of $M$. Now $M$ is said to be $GV$-\emph{torsion} (resp., $GV$-\emph{torsion-free}) if ${\rm tor}_{GV}(M) =M$ (resp., ${\rm tor}_{GV}(M) =0$).   A $GV$-torsion-free module $M$ is called a $w$-\emph{module} if ${\rm Ext}^1_R(R/J, M) =0$ for any $J\in GV(R)$. Projective modules and reflexive modules are $w$-modules. In the recent paper \cite{WANG FLAT}, it was shown that flat modules are $w$-modules. The notion of $w$-modules was introduced firstly over a domain \cite{MCCASLAND} in the study of Strong Mori
	domains and was extended to commutative rings with zero divisors in \cite{HFX}. Let $w$-${\rm Max}(R)$ denote the set of maximal $w$-ideals of $R$, i.e., $w$-ideals of $R$ maximal among proper integral $w$-ideals of $R$. Following  \cite[Proposition 3.8]{HFX}, every maximal $w$-ideal is prime.
	For any $GV$-torsion-free module $M$,
	$$M_{w}:=\{x\in E(M)\mid Jx\subseteq M \text{  for some } J\in  GV(R)\}$$
	is a $w$-submodule of $E(M)$ containing $M$ and is called the $w$-\emph{envelope} of $M$,
	where $E(M)$ denotes the injective hull  of $M$. It is clear that a $GV$-torsion-free module $M$ is a $w$-module if and only if $M_{w}=M$.
	Let $M$ and $N$ be $R$-modules and let $f : M \rightarrow N$ be a homomorphism. Following \cite{FGTYPE},
	$f$ is called a $w$-\emph{monomorphism} (resp., $w$-\emph{epimorphism}, $w$-\emph{isomorphism}) if $f_{\mathfrak{p}} :
	M_{\mathfrak{}}\rightarrow N_{\mathfrak{p}}$ is a monomorphism (resp., an epimorphism, an isomorphism) for all
	$\mathfrak{p}\in w\text{-}{\rm Max}(R)$.  A sequence $0\rightarrow A \rightarrow B \rightarrow C\rightarrow 0$ of $R$-modules is said to be $w$-exact if $0\rightarrow A_{\p} \rightarrow B_{\p} \rightarrow C_{\p}\rightarrow 0$ is exact for any $\p\in$ $w$-${\rm Max}(R)$. An $R$-module $M$ is called a $w$-\emph{flat module} if the induced map $1 \otimes f : M\otimes A \rightarrow M\otimes B$ is a $w$-monomorphism for any $w$-monomorphism $f : A\rightarrow B$. Certainly flat modules are $w$-flat, but the converse implication is not true in general.\\
	Recall from \cite{MADDOX} that an $R$-module $A$ is called absolutely pure if $A$ is a pure submodule in every $R$-module which contains $A$ as a submodule.  C. Megibben showed in \cite{MEGIBBEN}, that an $R$-module $A$ is absolutely pure if and only if $\Ext_R^1(N,A)=0$ for every finitely presented $R$-module $N$. Hence, an absolutely pure module is precisely an $FP$-injective module in \cite{STENSTROM}. For more details about absolutely pure (or $FP$-injective) modules, see  \cite{MADDOX, MEGIBBEN, STENSTROM, JAIN, PINZON}. In a very recent paper\cite{SXFW}, the authors introduced the notion of absolutely $w$-pure modules as generalization of absolutely pure (\FP-injective) modules in the sense of the $w$-operation theory. As in \cite{SXFW1}, a $w$-exact sequence of $R$-modules $\xymatrix{0\rightarrow A\rightarrow B\rightarrow C\rightarrow 0}$ is said to be $w$-pure exact if, for any $R$-module $M$,
	the induced sequence $\xymatrix{0\rightarrow A\otimes M\rightarrow B\otimes M \rightarrow C\otimes M \rightarrow 0}$ is $w$-exact.  In particular,
	if $A$ is a submodule of $B$ and $\xymatrix{0\rightarrow A\rightarrow B\rightarrow B/A\rightarrow 0}$ is a $w$-pure exact sequence of
	$R$-modules, then $A$ is said to be a $w$-pure submodule of $B$.  If $A$ is a $w$-pure submodule in every $R$-module which contains $A$ as a submodule, then $A$ is said to be an absolutely $w$-pure module. Following \cite[Theorem 2.6]{SXFW}, an $R$-module $A$ is absolutely $w$-pure if and only if $\Ext_R^1(N,A)$  is a $GV$-torsion $R$-module for every finitely presented $R$-module $N$. In \cite{MAODING}, Ding and Mao introduced  and studied the notion of $\FP$-projective dimension of modules and rings;  the $\FP$-projective dimension of an $R$-module $M$, denoted by $\fpd_{R}(M)$, is the smallest positive integer $n$ for which $\Ext_R^{n+1}(M,A)=0$ for all absolutely pure ($\FP$-injective) $R$-modules $A$, and $\FP$-projective dimension of $R$, denoted by $\fpD(R)$, is defined as the supremum of the $\FP$-projective dimensions of finitely generated $R$-modules. These dimensions measures how far away a finitely generated module is from being finitely presented, and how far away a ring is from being Noetherian.\\
	In Section $2$, we introduce the concept of $w$-FP-projective modules. Hence, we prove that a ring  $R$ is $DW$ (\cite{AM})  if and only if every FP-projective $R$-module is $w$-FP-projective if and only if every finitely presented $R$-module is $w$-FP-projective, and $R$ is a coherent $DW$ ring if and only if every  finitely generated ideal is $w$-FP-projective.\\
	Section $3$ deals with  the $w$-FP projective dimension of modules and rings. After a routine study of these dimensions, we prove that  $R$ is a Noetherian $DW$ ring if and only if every $R$-module is $w$-FP-projective and $R$ is FP-hereditary $DW$-ring if and only if every submodule of projective $R$-module is $w$-FP-projective.
	
	\section{\bf w-FP-projective modules}
	We start with the following definition.
	\begin{definition}\label{w-FP-proj}
		An $R$-module $M$ is said to be $w$-$\FP$-projective if $\Ext_R^1(M,A)=0$ for any  absolutely $w$-pure $R$-module $A$.
	\end{definition}
	Since every absolutely pure  module is absolutely $w$-pure (\cite[Corollary 2.7]{SXFW}), we have the following inclusions:
	$$\{\textrm{Projective modules}\}\subseteq \{w\textrm{-}\FP\textrm{-projective modules}\}\subseteq \{\FP\textrm{-projective modules}\}$$
	Recall that a ring $R$ is called a $DW$-ring if every ideal of $R$ is a $w$-ideal, or equivalently every maximal ideal of $R$ is $w$-ideal \cite{AM}. Examples of $DW$-rings are Pr\"{u}fer domains, domains with Krull dimension one, and rings with Krull dimension zero. Hence, it is clear that if $R$ is a $DW$-ring, then $w$-$\FP$-projective $R$-modules are just the $\FP$-projective $R$-modules. Moreover, using \cite[Corollary 2.9]{SXFW}, it is easy to see that over a  von Neumann regular ring, the three classes of modules above coincide.\\
	\begin{remark}\label{ENOCHSFP}
		It is proved in \cite{ENOCHS} that a finitely generated $R$-module $M$ is finitely presented if and only if $\Ext_{R}^1(M,A)=0$ for any absolutely pure ($\FP$-injective) $R$-module $A$. Thus, every finitely generated $w\textrm{-}\FP$-projective $R$-module is finitely presented.
	\end{remark}
	We need the following lemma.

	\begin{lemma}\label{GVABS}
		Every $GV$-torsion $R$-module is absolutely $w$-pure.
	\end{lemma}
	\begin{proof}
		Let $A$ be an arbitrary $R$-module and $N$ be a finitely presented $R$-module. For any maximal $w$-ideal $\p$ of $R$, the naturel homomorphism
		$$ \theta:\Hom_{R}(N,A)_{\p}\rightarrow \Hom_{R_{\p}}(N_{\p},A_{\p})$$
		induces a homomorphism
		$$ \theta_{1}: \Ext_{R}^1(N,A)_{\p}\rightarrow \Ext_{R_{\p}}^1(N_{\p},A_{\p})$$
		Following \cite[Proposition 1.10]{FHT}, $\theta_{1}$ is a monomorphism.
		Suppose that $A$ is a $GV$-torsion $R$-module. Then, we get $(\Ext_{R}^1(N,A))_{\p} = 0$ since $A_{\p}=0$ (by \cite[Lemma 0.1]{FHT}).  Hence, $\Ext_{R}^1(N,A)$ is $GV$-torsion (by \cite[Lemma 0.1]{FHT}). Consequently, $A$ is an absolutely $w$-pure $R$-module (by \cite[Theorem 2.6]{SXFW}).
	\end{proof}
	The first main result of this paper characterizes $DW$-rings in terms of $w$-$\FP$-projective $R$-modules.
	\begin{proposition}\label{DW}
		Let $R$ be a ring. Then the following conditions are equivalent:
		\begin{itemize}
			\item [(1)] Every finitely presented $R$-module is $w$-$\FP$-projective.
			\item [(2)] Every $\FP$-projective $R$-module is $w$-$\FP$-projective.
			\item [(3)] $R$ is a $DW$-ring.
		\end{itemize}
	\end{proposition}
	\begin{proof}
		$(3)\Rightarrow (2)$. It is obvious and $(2)\Rightarrow (1)$ follows from the fact that  finitely presented $R$-modules are always FP-projective.\\
		$(1)\Rightarrow (3)$. Suppose that $R$ is not a $DW$-ring. Then, by \cite[Theorem 6.3.12]{KIMBOOK}, there exist maximal ideal $\m$ of $R$ which is not $w$-ideal, and so by \cite[Theorem 6.2.9]{KIMBOOK}, $\m_{w}=R$. Hence, by \cite[Proposition 6.2.5]{KIMBOOK}, $R/\m$ is a $GV$-torsion $R$-module (sine $\m$ is  a $GV$-torsion-free $R$-module), and so $R/\m$ is an absolutely $w$-pure $R$-module (by Lemma   \ref{GVABS}). Hence, by hypothesis, for any $I$ finitely generated ideal $I$ of $R$, we get $\Ext_{R}^1(R/I,R/\m)=0$. Using \cite[Lemma 3.1]{DING4}, we obtain that $\Tor_{R}^1(R/I,R/\m)=0$, which means that $R/\m$ is flat. Accordingly, $\m$ is a $w$-ideal, and then $\m_{w}=\m$,
		a contradiction with $\m_{w} = R$. Consequently, $R$ is a $DW$-ring.
	\end{proof}
	\begin{remark}\label{example1}
		Let $(R, \m)$ be a regular local ring with ${\rm gldim}(R) =n$ ($n\geq 2$).  By \cite[Example 2.6]{MRM}, $R$ is not $DW$ ring. Hence, there exists an FP-projective $R$-module $M$ which is not $w$-FP-projective.
	\end{remark}

	Next, we give some characterizations of $w$-$\FP$-projective modules.
	\begin{proposition}
		Let $M$ be an $R$-module. Then the following are equivalent:
		\begin{enumerate}
			\item $M$ is $w$-$\FP$-projective.
			\item $M$ is projective with respect to every exact sequence $\xymatrix{ 0\rightarrow A\rightarrow B\rightarrow C\rightarrow 0 }$, where $A$ is absolutely $w$-pure.
			\item $P\otimes M$ is $w$-$\FP$-projective for any  projective $R$-module $P$.
			\item $\Hom(P,M)$ is $w$-$\FP$-projective for any finitely generated projective $R$-module $P$.
		\end{enumerate}
	\end{proposition}
	\begin{proof}
		$(1)\Leftrightarrow(2)$. It is straightforward.\\
		$(1)\Rightarrow (3)$. Let $A$ be any absolutely $w$-pure $R$-module and $P$ be a projective $R$-module. Following \cite[Theorem 3.3.10]{KIMBOOK},  we have the  isomorphism:
		$$\Ext_{R}^1(P\otimes M,A)\cong \Hom(P,\Ext_{R}^1(M,A)).$$
		Since $M$ is $w$-$\FP$-projective, we have $\Ext_{R}^1(M,A)=0$. Thus, $\Ext_{R}^1(P\otimes M,A)=0$, and so $P\otimes M$ is $w$-$\FP$-projective.\\
		$(1)\Rightarrow (4)$. Let $A$ be any absolutely $w$-pure $R$-module and $P$ be a finitely generated projective $R$-module. Using \cite[Theorem 3.3.12]{KIMBOOK}, we have the  isomorphism:
		$$\Ext_{R}^1(\Hom(P,M),A)\cong P\otimes \Ext_{R}^1(M,A)=0.$$
		Hence, $\Hom(P,M)$ is a $w$-$\FP$-projective $R$-module.\\
		$(3)\Rightarrow (1)$ and $(4)\Rightarrow (1)$.  Follow by letting $P=R$.
	\end{proof}
	Recall that a fractional ideal $I$ of a domain $R$ is said to be $w$-invertible if $(II^{-1})_{w}=R$. A domain $R$ is said to be a Pr\"{u}fer $v$-multiplication domain ($PvMD$) when any nonzero finitely generated ideal of $R$ is $w$-invertible. Equivalently, $R$ is a $PvMD$ if and only if $R_{\p}$ is a valuation domain for any maximal $w$-ideal $\p$ of $R$ (\cite[Theorem 2.1]{WANGMC}). The class of $PvMDs$ strictly contains the classes of Pr\"{u}fer domains, Krull domains, and integrally closed coherent domains.
	\begin{proposition}
		Let $R$ be a $PvMD$. Then $\pd_{R}(M)\leq 1$ for any $w$-$\FP$-projective $R$-module $M$.
	\end{proposition}
	\begin{proof}
		Let $M$ be a $w$-$\FP$-projective $R$-module. Following  \cite[Theorem 2.10]{SXFW}, every $h$-divisible $R$-module is absolutely $w$-pure. Hence,  $\Ext_{R}^1(M,D)=0$ for any $h$-divisible $R$-module $D$. Hence, by \cite[\romannumeral 7, Proposition 2.5]{FUCHSSALCE}, $\pd_{R}(M)\leq 1$, as desired.
	\end{proof}
	\begin{proposition}
		If $M$ is a $w$-$\FP$-projective $R$-module and $\Ext_{R}^1(M,G)=0$ for any $GV$-torsion-free $R$-module $G$, then $M$ is projective.
	\end{proposition}
	\begin{proof}
		Let $A$ be an arbitrary $R$-module. The exact sequence
		$$\xymatrix{ 0\rightarrow {\rm tor}_{GV}(A)\rightarrow A\rightarrow A/{\rm tor}_{GV}(A)\rightarrow 0}$$
		gives rise to the exact sequence
		$$\xymatrix{ 0=\Ext_{R}^1(M,{\rm tor}_{GV}(A))\rightarrow \Ext_{R}^1(M,A)\rightarrow \Ext_{R}^1(M, A/{\rm tor}_{GV}(A))=0}$$
		Thus $\Ext_{R}^1(M,A)=0$, and so $M$ is projective.
	\end{proof}
	
	\begin{proposition}
		Let $(R,\m)$ be a local ring which not $DW$-ring (for example,  regular local rings $R$ with ${\rm gldim}(R) =n$ ($n\geq 2$)). Then every finitely generated $w$-$\FP$-projective $R$-module $M$ is free.
	\end{proposition}
	\begin{proof}Let $M$ be a finitely generated $w$-$\FP$-projective $R$-module. As in the proof of Proposition  \ref{DW}, we obtain that $\Tor_{R}^1(M,R/\m)=0$. But $M$ is finitely generated, and  so finitely presented (by Remark   \ref{ENOCHSFP}). Hence, by \cite[Lemma 2.5.8]{SG},  $M$ is projective. Consequently, $M$ is free since $R$ is local.
	\end{proof}
	\begin{proposition}
		The class of all $w\textrm{-}\FP$-projective modules is closed under arbitrary direct sums and under direct summands.
	\end{proposition}
	\begin{proof}
		Follows  from \cite[Theorem 3.3.9(2)]{KIMBOOK}.
	\end{proof}
	Recall that a ring $R$ is called  coherent if every finitely generated ideal of $R$ is finitely presented.
	\begin{lemma}\label{STRONGFP}
		Let $R$ be a coherent ring and  $A$ be an $R$-module. Then $A$ is absolutely $w$-pure if and only if  $\Ext_{R}^{n+1}(N,A)$ is a $GV$-torsion $R$-module for any finitely presented module $N$ and any integer $n\geq 0$.
	\end{lemma}
	\begin{proof} $(\Rightarrow)$
		Suppose that $A$ is absolutely $w$-pure $R$-module and let $N$ be a finitely presented $R$-module. The case $n=0$ is obvious. Hence, assume that   $n> 0$. Consider an exact sequence
		$$\xymatrix{ 0\rightarrow N' \rightarrow F_{n-1}\rightarrow \cdots \rightarrow F_{0}\rightarrow N\rightarrow 0 }$$
		where $ F_{0}$,...,$F_{n-1}$ are finitely generated free $R$-module and  $N'$ is finitely presented. Such sequence exists  since $R$ is  coherent. Thus, $(\Ext_{R}^{n+1}(N,A))_{\p}\cong (\Ext_{R}^{1}(N',A))_{\p}=0$ for any $w$-maximal ideal $\p$ of $R$. So,  $\Ext_{R}^{n+1}(N,A)$ is a $GV$-torsion $R$-module.  \\
		$(\Leftarrow)$ Clear.
	\end{proof}
	\begin{lemma}\label{exactpure}
		Let $R$ be a coherent ring and  $\xymatrix{ 0\rightarrow A\rightarrow B\rightarrow C\rightarrow 0}$ be an exact sequence of $R$-modules, where $A$ is
		absolutely $w$-pure. Then, $B$ is absolutely $w$-pure if and only if $C$ is absolutely $w$-pure.
	\end{lemma}
	\begin{proof}
		Let $N$ be a finitely presented $R$-module. We have
		$$\xymatrix{ \Ext_{R}^{1}(N,A)\rightarrow \Ext_{R}^{1}(N,B)\rightarrow \Ext_{R}^{1}(N,C)\rightarrow \Ext_{R}^{2}(N,A)}.$$
		By  Lemma  \ref{STRONGFP}, for any maximal $w$-ideal $\p$, we get
		$$\xymatrix{ 0=\Ext_{R}^{1}(N,A)_{\p}\rightarrow \Ext_{R}^{1}(N,B)_{\p}\rightarrow \Ext_{R}^{1}(N,C)_{\p}\rightarrow \Ext_{R}^{2}(N,A)_{\p}}=0.$$
		Thus, $\Ext_{R}^{1}(N,B)_{\p}\cong \Ext_{R}^{1}(N,C)_{\p}$. So, $\Ext_{R}^{1}(N,B)$ is a $GV$-torsion $R$-module if and only if $\Ext_{R}^{1}(N,C)$ is a $GV$-torsion $R$-module. Thus, $B$ is absolutely $w$-pure if and only if $C$ is absolutely $w$-pure.
	\end{proof}
	\begin{proposition}\label{SFP}
		Let $R$ be a coherent ring and $M$ be an $R$-module. Then the following are equivalent:
		\begin{enumerate}
			\item $M$ is $w$-$\FP$-projective.
			\item $\Ext_{R}^{n+1}(M,A)=0$ for any absolutely $w$-pure module $A$ and any integer $n\geq 0$.
		\end{enumerate}
	\end{proposition}
	\begin{proof}
		$(1)\Rightarrow(2)$. Let $A$ be an absolutely $w$-pure $R$-module. The case $n=0$ is obvious. So, we may assume $n>0$. Consider an exact sequence
		$$\xymatrix{ 0\rightarrow A \rightarrow E^{0}\rightarrow \cdots \rightarrow E^{n-1}\rightarrow A^{'}\rightarrow 0 }$$
		where $ E^{0}$,...,$E^{n-1}$ are injective $R$-modules. By Lemma   \ref{exactpure}, $A'$ is absolutely $w$-pure. Hence, $\Ext_{R}^{n+1}(M,A)\cong \Ext_{R}^{1}(M,A')=0$.\\
		$(2)\Rightarrow(1)$. Obvious.
	\end{proof}
	\begin{proposition}
		Let $R$ be a coherent ring and  $\xymatrix{ 0\rightarrow M''\rightarrow M'\rightarrow M\rightarrow 0}$ be an exact sequence of $R$-modules, where $M$ is
		$w$-$\FP$-projective. Then, $M'$ is $w$-$\FP$-projective if and only if $M''$ is $w$-$\FP$-projective.
	\end{proposition}
	\begin{proof}
		Follows from  Proposition   \ref{SFP}.
	\end{proof}
	We end this section with the following characterizations of a coherent $DW$-rings.
	\begin{proposition}\label{COHERENTDW}
		Let $R$ be a ring. Then the following are equivalent:
		\begin{enumerate}
			\item $R$ is a coherent $DW$-ring.
			\item Every finitely generated submodule of a projective $R$-module is $w$-$\FP$-projective.
			\item Every finitely generated ideal of $R$ is $w$-$\FP$-projective.
		\end{enumerate}
	\end{proposition}
	\begin{proof}
		$(1)\Rightarrow(2)$. Follows immediately from \cite[Theorem 3.7]{MORSH} since, over a $DW$-ring, the classes of $w$-$\FP$-projective modules and $\FP$-projective modules coincide.\\
		$(2)\Rightarrow(3)$. Obvious.\\
		$(3)\Rightarrow (1)$. $R$ is coherent by Remark  \ref{ENOCHSFP}.  Assume that $R$ is not a $DW$-ring. As in the proof of Proposition  \ref{DW}, there exist a maximal ideal $\m$ of $R$ such that $R/\m$ is   absolutely $w$-pure and $\m_{w}=R$. So, for any finitely generated ideal $I$ of $R$, we have
		$$\xymatrix{ 0=\Ext_{R}^1(I,R/\m)\rightarrow \Ext_{R}^2(R/I,R/\m) \rightarrow \Ext_{R}^2(R,R/\m)=0},$$
		and then $\Ext_{R}^2(R/I,R/\m)=0$.
		By  \cite[Lemma 3.1]{DING4}, $\Tor_{R}^2(R/I,R/\m)=0$, which means that $\fd_{R}(R/\m)\leq 1$.
		Then, $\m$ is flat, and so a $w$-ideal, a contradiction.
	\end{proof}
	\begin{corollary}\label{coro3}
		Let $R$ be a domain. Then $R$ is a coherent $DW$-domain if and only if every finitely generated torsion-free $R$-module is $w$-$\FP$-projective.
	\end{corollary}
	\begin{proof} Following \cite[Theorem 1.6.15]{KIMBOOK}, every  finitely generated
		torsion-free $R$-module can be embedded in a finitely generated free module (since $R$ is a domain). Hence, 	$(\Rightarrow)$ follows immediately from  Proposition  \ref{COHERENTDW}. For 	$(\Leftarrow)$, it suffices to see that since $R$ is a domain, every ideal is torsion-free, and then use Proposition  \ref{COHERENTDW}.
	\end{proof}
	\section{\bf The w-FP-projective dimension of modules and rings}
	In this section, we introduce and investigate the $w$-FP-projective dimension of modules and rings.
	\begin{definition}\label{defdim}
		Let $R$ be a ring. For any $R$-module $M$, the $w$-$\FP$-projective dimension of $M$, denoted by $w$-$\fpd_R(M)$, is the smallest integer $n\geq0$ such that $\Ext_R^{n+1}(M,A)=0$  for any absolutely $w$-pure $R$-module $A$. If no such integer exists, set $w$-$\fpd_R(M)=\infty$.\\
		The $w$-FP-projective dimension of $R$ is defined by:
		$$w\textrm{-}\fpD(R)=\sup\{w\textrm{-}\fpd_R(M) :\textrm{ M is finitely generated } R\textrm{-module}\}$$
	\end{definition}
	Clearly, an $R$-module $M$ is  $w$-$\FP$-projective if and only if $w\textrm{-}\fpd_R(M)=0$, and $\fpd_R(M)\leq w\textrm{-}\fpd_R(M)$, with equality when
	$R$ is a $DW$-ring. However, this inequality may be strict (Remark  \ref{example1}). Also, $\fpD(R)\leq w\textrm{-}\fpD(R)$ with equality when
	$R$ is a $DW$-ring, and  this inequality may be strict.   To see that, consider 	a regular local ring $(R, \m)$ with ${\rm gldim}(R) =n$ ($n\geq 2$). Since $R$ is Noetherian, we get  $\fpD(R)=0$ (by \cite[Proposition 2.6]{MAODING}). Moreover, by  Remark  \ref{example1}, there exists an (FP-projective) $R$-module $M$ which is not $w$-FP-projective. Thus, $w\textrm{-}\fpD(R)>0$.   \\
	First, we give a description of the $w$-FP-Projective dimension of modules over coherent ring.
	
	\begin{proposition}\label{propo13}
		Let $R$ be a coherent ring. The following statements are equivalent for an $R$-module $M$.
		\begin{enumerate}
			\item $w\textrm{-}\fpd(M)\leqslant n$.
			\item $\Ext_R^{n+1}(M,A)=0$ for any absolutely $w$-pure $R$-module $A$.
			\item $\Ext_R^{n+j}(M,A)=0$ for any absolutely $w$-pure $R$-module $A$ and any $j\geq 1$.
			\item If the sequence $0\rightarrow P_n\rightarrow P_{n-1}\rightarrow\cdots\rightarrow P_0\rightarrow M\rightarrow 0$ is exact with $P_0,\cdots,P_{n-1}$ are $w\textrm{-}\FP$-projective $R$-modules, then $P_n$ is $w\textrm{-}\FP$-projective.
			\item If the sequence $0\rightarrow P_n\rightarrow P_{n-1}\rightarrow\cdots\rightarrow P_0\rightarrow M\rightarrow 0$ is exact with $P_0,\cdots,P_{n-1}$ are projective $R$-modules, then $P_n$ is $w\textrm{-}\FP$-projective.
			\item There exists an exact sequence $0\rightarrow P_n\rightarrow P_{n-1}\rightarrow\cdots\rightarrow P_0\rightarrow M\rightarrow 0$ where each $P_{i}$
			is $w\textrm{-}\FP$-projective.
		\end{enumerate}
	\end{proposition}
	\begin{proof}
		$(3)\Rightarrow (2)\Rightarrow (1)$ and $(4)\Rightarrow (5)\Rightarrow (6)$. Trivial.\\
		$(1)\Rightarrow (4)$. Let $0\rightarrow P_n\rightarrow P_{n-1}\rightarrow\cdots\rightarrow P_0\rightarrow M\rightarrow 0$ be an exact sequence of $R$-modules with $P_{0}, \cdots, P_{n-1}$ are $w\textrm{-}\FP$-projective, and set $K_{0}=\Ker(P_{0}\rightarrow M)$ and $K_{i}=\Ker(P_{i}\rightarrow P_{i-1})$, where $i=1,\cdots,n-1$. Using Proposition \ref{SFP}, we get 	
		$$0=\Ext_R^{n+1}(M,A)\cong \Ext_R^{n}(K_0,A) \cong \cdots\cong\Ext_R^{1}(P_{n},A)$$
		for all absolutely $w$-pure $R$-module $A$. Thus, $P_{n}$ is $w\textrm{-}\FP$-projective.\\
		$(6)\Rightarrow(3)$. We procced by induction on $n\geq0$. For the $n=0$, $M$ is $w\textrm{-}\FP$-projective module and so $(3)$ holds by proposition \ref{SFP}. If $n\geq1$,  then there is an eact sequence $0\rightarrow P_n\rightarrow P_{n-1}\rightarrow\cdots\rightarrow P_0\rightarrow M\rightarrow 0$ where each $P_{i}$ is $w\textrm{-}\FP$-projective. Set $K_{0}=\Ker(P_{0}\rightarrow M)$. Then, we have the following exact sequences
		$$0\rightarrow P_n\rightarrow P_{n-1}\rightarrow\cdots\rightarrow P_{1}\rightarrow K_{0}\rightarrow 0$$
		and
		$$0\rightarrow K_{0}\rightarrow P_{0}\rightarrow M\rightarrow 0$$
		Hence, by induction $\Ext_R^{n-1+j}(K_{0},A)=0$ for all absolutely $w$-pure $R$-module $A$ and all $j\geq1$. Thus, $\Ext_R^{n+j}(M,A)=0$, and so we have the desired result.
	\end{proof}
	The proof of the next proposition is standard homological algebra. Thus we omit its proof.
	\begin{proposition}
		Let $R$ be a coherent ring and  $0\rightarrow M''\rightarrow M'\rightarrow M\rightarrow 0$ be an exact sequence of $R$-modules.
		If two of $w\textrm{-}\fpd_R(M'')$, $w\textrm{-}\fpd_R(M')$ and $w\textrm{-}\fpd_R(M)$ are finite, so is the third. Moreover
		\begin{enumerate}
			\item $w\textrm{-}\fpd_R(M'')\leq \sup\left\{w\textrm{-}\fpd_R(M'),\;w\textrm{-}\fpd_R(M)-1\right\}$.
			\item $w\textrm{-}\fpd_R(M')\leq \sup\{w\textrm{-}\fpd_R(M''),\;w\textrm{-}\fpd_R(M)\}$.
			\item $w\textrm{-}\fpd_R(M)\leq \sup\{w\textrm{-}\fpd_R(M'),\;w\textrm{-}\fpd_R(M'')+1\}$.
		\end{enumerate}
	\end{proposition}
	\begin{corollary}\label{coro7}
		Let $R$ be a coherent ring and $0\rightarrow M''\rightarrow M'\rightarrow M\rightarrow 0$ be an exact sequence of $R$-modules. If $M'$ is $w$-$\FP$-projective and $w$-$\fpd_R(M)>0$, then $w\textrm{-}\fpd_R(M)=w\textrm{-}\fpd_R(M'')+1$.
	\end{corollary}
	
	\begin{proposition}\label{propo5}
		Let $R$ be a coherent ring and $\{M_{i}\}$ be a family of $R$-modules. Then $w\textrm{-}\fpd_R(\oplus_{i}M_{i})=\sup_{i}\{w\textrm{-}\fpd_R(M_{i})\}$.
	\end{proposition}
	\begin{proof}
		The proof is straightforward.
	\end{proof}
	
	\begin{proposition}\label{DIMENSION}
		Let $R$ be a ring and $n\geq 0$ be a an integer. Then the following statements are equivalent:
		\begin{enumerate}
			\item $w\textrm{-}\fpD(R)\leq n$.
			\item $w\textrm{-}\fpd(M)\leqslant n$ for all $R$-modules $M$.
			\item $w\textrm{-}\fpd(R/I)\leqslant n$ for all ideals $I$ of $R$.
			\item $\id_{R}(A)\leqslant n$ for all absolutely $w$-pure $R$-modules $A$.
		\end{enumerate}
		Consequently, we have
		\begin{align*}
		w\textrm{-}\fpD(R)&=\sup \{w\textrm{-}\fpd_R(M) \mid M~\textrm{is an}~R\textrm{-module}\}\\
		&=\sup \{w\textrm{-}\fpd_R(R/I) \mid I~\textrm{is an ideal of}~R \}\\
		&=\sup \{\id_R(A) \mid A~\textrm{is an abosolutely}~w\textrm{-pure}~R\textrm{-module}\}
		\end{align*}
		
	\end{proposition}
	\begin{proof}
		$(2)\Rightarrow(1)\Rightarrow(3)$. Trivial.\\
		$(3)\Rightarrow(4)$. Let $A$ be an absolutely $w$-pure $R$-module. For any ideal $I$ of $R$, we have $\Ext_R^{n+1}(R/I,A)=0$. Thus, $\id_{R}(A)\leqslant n$.\\
		$(4)\Rightarrow(2)$. Let $M$ be an $R$-module. For any absolutely $w$-pure $R$-module $A$, we have  $\Ext_R^{n+1}(M,A)=0$. Hence, $w\textrm{-}\fpd(M)\leqslant n$.
	\end{proof}
	Note that Noetherian rings need not to be $DW$ (for example, a regular ring with global dimension 2), and $DW$-rings need not to be Noetherian (for example, a non Noetherian von Neumann regular ring). Next, we show that rings $R$ with $w\textrm{-}\fpD(R)=0$ are exactly Noetherian $DW$-rings.
	\begin{proposition}
		Let $R$ be a ring. Then the following are equivalent:
		\begin{enumerate}
			\item $w\textrm{-}\fpD(R)=0$.
			\item Every $R$-module is $w$-$\FP$-projective.
			\item $R/I$ is $w$-$\FP$-projective for every ideal $I$ of $R$.
			\item Every absolutely $w$-pure $R$-module is injective.
			\item $R$ is Noetherian $DW$-ring.
		\end{enumerate}
	\end{proposition}
	\begin{proof}
		The equivalence of (1), (2), (3), and (4) follows from Proposition  \ref{DIMENSION}.\\
		$(2)\Leftrightarrow(5)$. Follows from  Proposition \ref{DW} and \cite[Proposition 2.6]{MAODING}.
	\end{proof}
	Recall from\cite{MORSH}, that a ring $R$ is said  $FP$-hereditary if every ideal of $R$ is $FP$-projective. Note that $FP$-hereditary rings need not to be $DW$ (for example, a non $DW$ Noetherian ring), and $DW$-rings need not to be $FP$-hereditary (for example, a non coherent $DW$-ring). Next, we show that rings $R$ with $w\textrm{-}\fpD(R)\leq 1$ are exactly $FP$-hereditary $DW$-rings.
	\begin{proposition}
		Let $R$ be a ring. Then the following are equivalent:
		\begin{enumerate}
			\item $w\textrm{-}\fpD(R)\leq 1$.
			\item Every submodule of $w$-$\FP$-projective $R$-module is $w$-$\FP$-projective.
			\item Every submodule of projective $R$-module is $w$-$\FP$-projective.
			\item $I$ is $w$-$\FP$-projective  for every ideal $I$ of $R$.
			\item $\id_{R}(A)\leq 1$ for all absolutely $w$-pure $R$-module $A$.
			\item $R$ is  a (coherent) $FP$-hereditary $DW$-ring.
		\end{enumerate}
	\end{proposition}
	\begin{proof} The implications $(2)\Rightarrow (3)\Rightarrow (4)$. are obvious. \\
		$(1)\Leftrightarrow (5)$. By Proposition \ref{DIMENSION}.\\
		$(4)\Rightarrow (5)$. Let $A$ be an absolutely $w$-pure $R$-module and $I$ be an ideal of $R$. The exact sequence $0\rightarrow I\rightarrow R\rightarrow R/I\rightarrow 0$ gives rise to the exact sequence
		$$0=\Ext_R^{1}(I,A)\rightarrow \Ext_R^{2}(R/I,A)\rightarrow \Ext_R^{2}(R,A)=0.$$
		Thus, $\Ext_R^{2}(R/I,A)=0$, and so $\id_{R}(A)\leq 1$.\\
		$(5)\Rightarrow (4)$. Let $I$ be an ideal of $R$. For any absolutely $w$-pure $R$-module $A$, we have
		$$0=\Ext_R^{2}(R/I,A)=\Ext_R^{1}(I,A).$$
		Thus, $I$ is $w$-$\FP$-projective.\\
		$(4)\Rightarrow (6)$. By hypothesis, $R$ is  $FP$-hereditary. Now, by
		Proposition \ref{COHERENTDW} $R$ is a coherent $DW$-ring.\\
		$(6)\Rightarrow(2)$. By \cite[Theorem 3.16]{MORSH}, since the $w$-$\FP$-projective $R$-modules are just the $\FP$-projective $R$-modules over a $DW$-ring.\\
		
	\end{proof}

\end{document}